\documentclass[12pt,a4paper,reqno]{amsproc}
\usepackage{amsmath,amsthm,wasysym,cite}
\usepackage{graphicx}
\usepackage{amsmath,amssymb,amsthm,cite}
\usepackage{enumerate}
\usepackage[mathscr]{eucal}
\usepackage{graphicx}
\usepackage{subfig}
\usepackage{cite,float}

\numberwithin{equation}{section}
\newtheorem{theorem}{Theorem}[section]
\newtheorem{lemma}[theorem]{Lemma}

\theoremstyle{remark}

\DeclareMathOperator{\RE}{Re} 
\begin{document}

\title[First Order Differential Subordination]{First Order Differential Subordination for Functions with Positive Real Part}

\author[O. P. Ahuja]{Om P. Ahuja}
\address{Department of Mathematics, Kent State University, Burton, USA}
\email{oahuja@kent.edu}

\author[S. Kumar]{Sushil Kumar}
\address{Bharati Vidyapeeth's College of Engineering,  Delhi--110063, India}
\email{sushilkumar16n@gmail.com}

\author{V. Ravichandran}
\address{Department of Mathematics, University of Delhi, Delhi--110007, India}
\email{vravi@maths.du.ac.in; vravi68@gmail.com}

\begin{abstract}
Sharp estimates  on $\beta$ are determined  so that an analytic function $p$ defined  on the open unit disk in the complex plane normalized by $p(0)=1$ is subordinate to some well known starlike functions with positive real part whenever  $1+\beta z p'(z), \,\,1+\beta z p'(z)/p(z), \,\,\mbox{or}\,\,1+\beta z p'(z)/p^{2}(z)$ is subordinate to $\sqrt{1+z}$.
Our results provide sharp version of previously known results.
\end{abstract}



\maketitle

\section{Introduction}\label{sec1}
Let $\mathcal{A}$ denote the class of analytic functions $f$ on  the disk $\mathbb{D}:= \{ z\in \mathbb{C}: |z|<1\}$ and  normalized by the  condition $f(0)=0=f'(0)-1$. Let $\mathcal{S}$ be the subset of $\mathcal{A}$ of univalent functions. An analytic function $f$ defined on $\mathbb{D}$ is subordinate to the analytic function $g$  on $\mathbb{D}$ (or $g$ is superordinate to $f$), if there exists an analytic function $w: \mathbb{D}\to  \mathbb{D}$, with $w(0)=0$, such that $f=g\circ w$.  Furthermore, if $g$ is univalent in $\mathbb{D}$, then $f \prec g$ is equivalent to $f(0) = g(0)$ and $f(\mathbb{D})\subseteq g(\mathbb{D})$, see \cite{o}. Let $p$ be an analytic function on  $\mathbb{D}$ normalized by $p(0)=1$. Goluzin \cite{Goluzin} discussed the  first order differential subordination $zp'(z) \prec zq'(z)$ and proved that, whenever $zq'(z)$ is convex,  the subordination $p(z)\prec q(z)$ holds   and the  function $q$ is best dominant. After this  basic result, many authors established several generalizations of first order differential subordination. The general theory of  differential subordination is discussed in the monograph by Miller and Mocanu \cite{Miller85}.

In 1989, Nunokawa \emph{et al.}\ \cite{p} proved that if subordination $1+z p'(z)\prec 1+z$ holds, then subordination $p(z)\prec 1+z$ also holds. In 2007,  Ali \emph{et al.}\ \cite{q} extended this result and determined the estimates on $\beta$ for which the subordination $1+ \beta z p'(z)/p^j(z)\prec (1+Dz)/(1+Ez)$ $(j=0,1,2)$ implies the subordination $p(z)\prec (1+Az)/(1+Bz)$, where $A,B,D,E \in [-1,1]$. In 2013, Omar  and  Halim \cite{r} determined the condition on $\beta$ in terms of complex number $D$ and real $E$ with $-1<E<1$ and $|D| \leq 1$ such that $1+\beta zp'(z)/p^j(z) \prec (1+Dz)/(1+Ez)$ $(j=0,1,2)$ implies $p(z) \prec \sqrt{1+z}$. Recently, Kumar and Ravichandran  \cite{sr} determined some sufficient conditions for certain first order differential subordinations to imply that the corresponding  analytic solution  is subordinate to a rational, exponential, or sine  function. For more details,  see \cite{vrk, sharma, nikol,cho}. The function $\sqrt{1+z}$ is associated with the class $\mathcal{S}_{L}^*$, introduced by Sok\'o\l\ and Stankiewicz \cite{b}. This class consists of the function $f \in \mathcal{A}$ such that $w(z):=zf'(z)/f(z)$ lies in the region bounded by the right half  of the lemniscate of Bernoulli given by  $|w^2-1|<1$. The lemniscate of Bernoulli is a best known plane curve resembling the symbol $\infty$. It was named after James Bernoulli who considered it in
elasticity theory in 1694. In geometry, the lemniscate is a plane curve defined by two given points $F_1$ and $F_2$, known as foci, at distance $2a$ from each other as the locus of points P so that $P F_1 . P F_2=a^2$. The equation of lemniscate may be written as $(x^2+y^2)^2=2a^2(x^2-y^2)$. The lemniscate in the complex plane is the locus of $z=x+iy$ such that $|z^2-a^2|=a^2$.

For an analytic function $p(z)=1+c_1z+c_2z^2+\cdots$, we determine the sharp bound on $\beta$  so that
 $p(z)\prec \mathcal{P}(z)$ where $\mathcal{P}(z)$ is a function with positive real part like $\sqrt{1+z}$, $(1+Az)/(1+Bz)$, $ e^z$, 
 $\varphi_0(z):=1+ \frac{z}{k}\left((k+z)/(k-z)\right)\,(k = \sqrt{2}+1)$, $\varphi_{sin}(z):= 1+\sin  z$, $\varphi_{C}(z):= 1+\frac{4}{3}z+\frac{2}{3}z^2$ and $\varphi_{\mbox{\tiny}\leftmoon}(z):=z+\sqrt{1+z^2}$, whenever $1+ \beta {z p'(z)}/{p^{j}(z)}\prec \sqrt{1+z},\,\,(j=0,1,2)$.
Many of our subordination results in this paper improve the corresponding non-sharp results obtained by earlier authors in \cite{c,svr,m}. Our results are sharp.

\section{Main Results}\label{sec2}
In 1985, Padmanabhan and Parvatham \cite{uni} introduced  a unified classes of starlike and convex functions using convolution  with the function of the form $z/(1-z)^{\alpha}$, $\alpha\in\mathbb{R}$.
Later, Shanmugam \cite{shan} considered the class $\mathcal{S}_g^*(h)$ of all $f\in{\mathcal{A}}$ satisfying $z(f*g)'/(f*g) \prec h$ where $h$ is a convex function, $g$ is a fixed function in $\mathcal{A}$. Denote by $\mathcal{S}^*(h)$ and $\mathcal{K}(h)$, the subclass $\mathcal{S}_g^*(h)$, when $g$ is $z/(1-z)$ and $z/(1-z)^2$ respectively.  In 1992, Ma and Minda \cite{a} considered  a weaker assumption that $h$ is a function with positive real part whose range is symmetric with respect to real axis and starlike with respect to $h(0)=1$ with $h'(0)>0$ and proved  distortion, growth, and covering theorems. The class  $\mathcal{S}^*(h)$ generalizes many subclasses of $\mathcal{A}$, for example,
$\mathcal{S^*}[A,\,B]:=\mathcal{S^*} ({(1+Az)}/{(1+Bz)})$ $ (-1\leq B<A\leq 1)$  \cite{j},  $\mathcal{S}_{L}^*:=\mathcal{S}^*(\sqrt{1+z})$  \cite{b}, $\mathcal{S}_e^*:=\mathcal{S}^*(e^z)$   \cite{m}, $\mathcal{S}^*_{sin}:=\mathcal{S}^*(\varphi_{sin}(z))$  \cite{sinf}, $\mathcal{S}^*_{C}:=\mathcal{S}^*(\varphi_{C}(z))$ \cite{afmc}, $ \mathcal{S}^*_{R}:=\mathcal{S}^*(\varphi_0(z))$ \cite{n},
and $\mathcal{S}^*_{\mbox{\tiny}\leftmoon}:=\mathcal{S}^*(\varphi_{\mbox{\tiny}\leftmoon}(z))$ \cite{lune1,lune2}.

Several sufficient conditions for functions to belong to the above defined classes can be obtained as an application of the following subordination results involving the   lemniscate of Bernoulli and other well known starlike functions with positive real part.
Our first result gives a bound on $\beta$ so that $1+\beta z p'(z) \prec \sqrt{1+z}$ implies that the function $p$ is subordinate to several well-known starlike functions.

\begin{theorem}\label{lemth1}
Let  the function $p$ be analytic in $\mathbb{D}$, $p(0)=1$ and $1+\beta z p'(z) \prec \sqrt{1+z}$.
Then the following subordination results hold:
\begin{itemize}
\item[(a)] If   $\beta \geq\frac{2(\sqrt{2}-1+\log 2-\log(1+\sqrt{2}))}{\sqrt{2}-1}\approx 1.09116$,  then  $p(z) \prec \sqrt{1+z}$.
\item[(b)] If  $\beta \geq \frac{2(1-\log{2})}{3-2\sqrt{2}}\approx 3.57694$,  then $p(z) \prec \varphi_0(z)$.
\item[(c)] If   $\beta \geq\frac{2 (1-\log 2)}{\sin (1)}\approx 0.729325$,   then $p(z) \prec \varphi_{sin}(z)$.
\item[(d)] If   $\beta \geq (2+\sqrt{2}) (1-\log 2)\approx 1.044766$,  then $p(z) \prec \varphi_{\mbox{\tiny}\leftmoon}(z)$.
\item[(e)]  If  $\beta \geq 3(1-\log{2})\approx 0.920558$,  then $p(z) \prec \varphi_C(z)$.
\item[(f)] Let  $-1<B<A<1$ and $B_0=\frac{2-\log 4-\sqrt{2}+\log(1+\sqrt{2})}{\sqrt{2}-\log(1+\sqrt{2}+1)}\approx 0.151764$. If either
\begin{itemize}
\item[(i)]  $B<B_0$ and $\beta \geq \frac{2(1-B)(1-\log 2)}{A-B}\approx 0.613706\frac{1-B}{A-B}$ or
\item[(ii)] $B>B_0$ and  $\beta \geq  \frac{2(1+B)(\sqrt{2}-1+\log 2-\log(1+\sqrt{2}))}{A-B}\approx0.451974\frac{1+B}{A-B}$, \\
 then $p(z) \prec {(1+Az)}/{(1+Bz)}$.
\end{itemize}
\end{itemize}
The bounds on $\beta$ are sharp.
\end{theorem}

In proving our results, the following lemma will be  needed.
\begin{lemma} \cite[Theorem 3.4h, p.\ 132]{o}\label{lemma1}
Let $q$ be analytic in $\mathbb{D}$ and let $\psi$ and  $\nu$ be analytic in a domain $U$ containing $q(\mathbb{D})$ with $\psi(w) \neq 0$ when $w \in q(\mathbb{D})$. Set
$Q(z):=zq'(z)\psi(q(z))$ and $h(z):=\nu(q(z))+Q(z)$.
Suppose that (i)  either $h$ is convex, or $Q$ is starlike univalent in $\mathbb{D}$ and (ii) $\RE \left({zh'(z)}/{Q(z)}\right) >0 $ for $z \in \mathbb{D}$.
If $p$ is analytic in $\mathbb{D}$, with  $p(0)=q(0)$, $p(\mathbb{D})\subseteq U$ and
\[\nu(p(z))+zp'(z) \psi(p(z)) \prec \nu(q(z))+zq'(z) \psi(q(z)),\]
then $p(z) \prec q(z)$, and $q$ is best dominant.
\end{lemma}

\begin{proof}[Proof of Theorem~\ref{lemth1}]
The function $q_{\beta}:\overline{\mathbb{D}} \to \mathbb{C}$  defined by
\[q_{\beta}(z)=1+\frac{2}{\beta} (\sqrt{1+z}-\log{(1+\sqrt{1+z})}+\log 2-1). \]
is  analytic and is a solution of the  differential equation $ 1+\beta z q_{\beta}'(z)=\sqrt{1+z}$. Consider the functions $\nu(w)=1$ and $\psi(w)=\beta$. The function $Q:\overline{\mathbb{D}}\to\mathbb{C}$ is defined by $Q(z)= z q_{\beta}'(z) \psi(q_{\beta}(z))=\beta z q_{\beta}'(z)$. Since $\sqrt{1+z}-1$ is starlike function in $\mathbb{D}$, it follows that function  $Q$ is starlike. Also note that the function $h(z)=\nu(q_{\beta}(z))+Q(z)$ satisfies $\RE (zh'(z)/Q(z)) > 0$ for  $z \in \mathbb{D}$. Therefore, by making use of  Lemma \ref{lemma1}, it follows that $1+ \beta  {zp'(z)} \prec 1+ \beta {z q_{\beta}'(z)}$ implies $p(z)\prec q_{\beta}(z)$.
Each of the conclusion in (a)-(f) is $p(z) \prec \mathcal{P} (z)$ for appropriate $\mathcal{P}$  and this holds if the subordination $q_{\beta}(z)\prec \mathcal{P} (z)$ holds. If $q_{\beta}(z)\prec \mathcal{P} (z)$, then  $\mathcal{P} (-1 )< q_{\beta}(-1)<q_{\beta}(1)< \mathcal{P} (1)$. This gives a necessary condition for $p \prec \mathcal{P} $ to hold. Surprisingly, this necessary condition is also sufficient. This can be seen by looking at the graph of the respective functions.

(a) On taking  $\mathcal{P}(z)=\sqrt{1+z}$, the inequalities  $q_{\beta}(-1)\geq0$ and $ q_{\beta}(1) \leq \sqrt{2}$ reduce to $\beta \geq \beta_1$ and $\beta \geq\beta_2$,
where
$\beta_1=2(1-\log 2)$  and  $\beta_2={2(\sqrt{2}-1+\log 2-\log(1+\sqrt{2}))}/{(\sqrt{2}-1)}$ respectively. Therefore,  the subordination $q_{\beta}(z)\prec \sqrt{1+z}$ holds only if
$\beta  \geq \max \left\{\beta_1, \,\,\beta_2\right\}=\beta_2$.

(b)
Consider $\mathcal{P}(z)=\varphi_0(z)$.
A simple calculation shows that the inequalities  $ q_{\beta}(-1)\geq \varphi_0(-1)$ and $ q_{\beta}(1) \leq \varphi_0(1)$  reduce to  $\beta \geq\beta_1$ and $\beta \geq \beta_2$, where $\beta_1= {2(1-\log 2)}/{(3-2\sqrt{2})}$ and $\beta_2=2(\sqrt{2}-1+\log 2-\log(1+\sqrt{2}))$ respectively.
Thus the subordination $q_{\beta}(z)\prec \varphi_0(z)$ holds only if $\beta \geq \max\{\beta_1, \beta_2\}=\beta_1$.

(c)
Consider $\mathcal{P}(z)=\varphi_{sin}(z)$. The  inequalities  $q_{\beta}(-1)\geq \varphi_{sin}(-1)$ and $q_{\beta}(1) \leq \varphi_{sin}(1)$ reduce to $\beta \geq \beta_1$ and  $\beta \geq\beta_2$,
where \[\beta_1=\frac{2(1-\log 2)}{\sin (1)}\,\,\mbox{and} \,\,\beta_2= \frac{2(\sqrt{2}-1+\log 2-\log(1+\sqrt{2}))}{\sin (1)}\] respectively. The subordination $q_{\beta}(z)\prec \varphi_{Sin}(z)$ holds if $\beta \geq \max\{\beta_1, \beta_2\}=\beta_1$.

(d)
Consider $\mathcal{P}(z)=\varphi_{\mbox{\tiny}\leftmoon}(z)$.
The  inequalities $ q_{\beta}(-1)\geq \varphi_{\mbox{\tiny}\leftmoon}(-1)$ and  $ q_{\beta}(1) \leq \varphi_{\mbox{\tiny}\leftmoon}(1)$ give $\beta \geq \beta_1$ and $ \beta \geq \beta_2$, where $\beta_1={(2+\sqrt{2})(1-\log 2)}$  and  $ \beta_2=\sqrt{2}(\sqrt{2}-1+\log 2-\log{(1+\sqrt{2})})$ respectively.
The subordination $q_{\beta}(z)\prec \varphi_{\mbox{\tiny}\leftmoon}(z)$ holds if $\beta \geq \max\{\beta_1, \beta_2\}=\beta_1$.

(e)
Consider $\mathcal{P}(z)=\varphi_C(z)$.
From the inequalities $\varphi_C(-1)\leq q_{\beta}(-1)$ and $q_{\beta}(1) \leq \varphi_C(1)$,  we get $\beta \geq 3(1-\log 2)$ and  $ \beta \geq 2(\sqrt{2}-1+\log 2-\log(1+\sqrt{2}))$ respectively. Thus the subordination $q_{\beta}(z)\prec \varphi_C(z)$ holds if
$\beta \geq \max \left\{3(1-\log 2), \,\,2(\sqrt{2}-1+\log 2-\log(1+\sqrt{2}))\right\}=3(1-\log 2)$.

(f)
Consider  $\mathcal{P}(z)={(1+Az)}/{(1+Bz)}$.
From the inequalities   $q_{\beta}(-1 )\geq {(1-A)}/{(1-B)}$ and $q_{\beta}(1) \leq  {(1+A)}/{(1+B)}$, we note that $\beta\geq \beta_1$ and $\beta \geq \beta_{2}$, where
\[\beta_1=\frac{2(1-B)(1-\log 2)}{A-B}\,\,\mbox{and} \,\,\beta_{2}=\frac {2(1+B)(\sqrt{2}-1+\log 2-\log(1+\sqrt{2}))}{A-B}\] respectively. A simple calculation gives $\beta_1-\beta_{2}=2(1-\log 2)+(1+B)(\log (1+\sqrt{2})-\sqrt{2})$. We note that $\beta_1-\beta_{2}\geq 0$ if $B<B_0$ and $\beta_1-\beta_{2}\leq 0$ if $B>B_0$ where
 \[B_0=\frac{2-\log 4-\sqrt{2}+\log(1+\sqrt{2})}{\sqrt{2}-\log(1+\sqrt{2}+1)}.\] The necessary subordination $p(z) \prec (1+Az)/(1+Bz)$ holds if  $\beta \geq \max \{\beta_1, \,\beta_{2}\}$.
\end{proof}

The subordination results in part (a) and (f) in Theorem \ref{lemth1} were also investigated by the authors in \cite[ Lemma 2.1, p.\ 1019]{c} and   \cite[Lemma 2.1, p.\ 3]{svr}, but their results were non-sharp.

Next result gives a bound on $\beta$ so that $1+\beta {z p'(z)}/{p(z)} \prec \sqrt{1+z}$ implies $p$ is subordinate to some well-known starlike functions.
\begin{theorem}\label{thd}
Let  the function $p$ be analytic in $\mathbb{D}$, $p(0)=1$ and  $1+\beta {z p'(z)}/{p(z)} \prec \sqrt{1+z}$.  Then the  following subordination results hold:
\begin{itemize}
\item[(a)] If  $\beta \geq \frac{2(\log 2-1)}{\log(2\sqrt{2}-2})\approx 3.26047$,  then $p(z) \prec \varphi_0(z)$.
\item[(b)] If   $\beta \geq \frac{2 \left(\sqrt{2}-1+\log (2)-\log \left(\sqrt{2}+1\right)\right)}{\log (1+\sin (1))}\approx 0.740256 $,  then $p(z)\prec \varphi_{sin}(z)$.
\item[(c)] If   $\beta \geq \frac{2 (\log 2 -1)}{\log (\sqrt{2}-1)} \approx 0.696306$,   then $p(z) \prec \varphi_{\mbox{\tiny}\leftmoon}(z)$.
\item[(d)] If  $\beta \geq 2(1-\log 2)\approx0.613706$,  then $p(z) \prec e^z$.
\item[(e)] If  $-1<B<A<1$ and $\beta \geq \max \{\beta_1, \,\beta_2\}$ where
\[\beta_1=\frac{2(1-\log 2)}{\log(1-B)-\log(1-A)} \,\,\,\mbox{and}\,\,\, \beta_2=\frac{2(\sqrt{2}-1+\log 2-\log(1+\sqrt{2}))}{\log(1+A)-\log(1+B)},\]
then $p(z) \prec {(1+Az)}/{(1+Bz)}$.
\end{itemize}
The bounds on $\beta$ are best possible.
\end{theorem}

\begin{proof}
The  function $q_{\beta}:\overline{\mathbb{D}} \to \mathbb{C}$ defined by
\[q_{\beta}(z)=\exp\left({\frac{2}{\beta} (\sqrt{1+z}-\log{(1+\sqrt{1+z})}+\log 2-1)}\right)\]
is  analytic and is a solution of  the differential equation $1+\beta z {q_{\beta}'(z)}/{q_{\beta}(z)} =\sqrt{1+z}$.
Define the functions $\nu(w)=1$ and $\psi(w)=\beta/w $. The function $Q:\overline{\mathbb{D}}\to\mathbb{C}$ defined by $Q(z):= z q_{\beta}'(z) \psi(q_{\beta}(z))=\beta z q_{\beta}'(z)/q_{\beta}(z)=\sqrt{1+z}-1$ is starlike in $\mathbb{D}$. The function  $h(z):= \nu(q_{\beta}(z))+Q(z)=1+Q(z)$ satisfies $\RE (zh'(z)/Q(z))>0$ for $ z \in\mathbb{D}$. Therefore, by using Lemma \ref{lemma1},  we see that the subordination
\[1+ \beta  \frac{zp'(z)}{p(z)} \prec 1+ \beta \frac{z q_{\beta}'(z)}{q_{\beta}(z)}\]
 implies $p(z)\prec q_{\beta}(z)$. As the similar  lines of the proof of Theorem \ref{lemth1}, the proofs of  parts (a)-(e) are completed.\qedhere
\end{proof}
The subordination in  part (d) and (e) of Theorem \ref{thd} were earlier investigated in \cite[Theorem 2.16(c), p.\ 10]{m} and \cite[Lemma 2.3, p.\ 5]{svr} where non-sharp results were obtained.

Next, we determine a bound on $\beta$ so that $1+\beta {z p'(z)}/{p^2(z)} \prec \sqrt{1+z}$ implies $p$ is subordinate to several well-known starlike functions.
\begin{theorem}\label{thd4}
Let  the function $p$ be analytic in $\mathbb{D}$, $p(0)=1$ and $1+\beta {z p'(z)}/{p^2(z)} \prec \sqrt{1+z}$.
Then the following subordination results hold for sharp bound of $\beta$:
\begin{itemize}
\item[(a)] If  $\beta \geq 4(1+\sqrt{2})(1-\log 2)\approx 2.96323$,  then $p(z) \prec \varphi_0(z)$.
\item[(b)] If   $\beta \geq \frac{2(1+\sin (1))(\sqrt{2}-\log(1+\sqrt{2})+\log 2-1)}{\sin (1)}\approx 0.989098$,
  then $p(z) \prec \varphi_{sin}(z)$.
\item[(c)] If   $\beta \geq (2+\sqrt{2})(\sqrt{2}-\log(1+\sqrt{2})+\log 2-1)\approx 0.771568$,   then $p(z) \prec \varphi_{\mbox{\tiny}\leftmoon}(z)$.
\item[(d)] Let  $-1<B<A<1$ and $A_0=\frac{2-\log 4-\sqrt{2}+\log(1+\sqrt{2})}{\sqrt{2}-\log(1+\sqrt{2}+1)}\approx 0.151764$. If either
\begin{itemize}
\item[(i)]  $A>A_0$  and $\beta \geq \frac{2(1-A)(1-\log 2)}{A-B}\approx 0.613706\frac{1-A}{A-B}$ or
\item[(ii)] $A<A_0$ and $\beta \geq \frac{2(1+A)(\sqrt{2}-1+\log 2-\log(1+\sqrt{2}))}{A-B}\approx0.451974\frac{1+A}{A-B}$, \\then $p(z) \prec {(1+Az)}/{(1+Bz)}$.
\end{itemize}
\end{itemize}
\end{theorem}

\begin{proof}
The  function $q_{\beta}:\overline{\mathbb{D}} \to \mathbb{C}$ defined by
\[q_{\beta}(z)=\left(1-\frac{2}{\beta}\left (\sqrt{1+z}-\log(1+\sqrt{1+z})+\log 2 -1\right)\right)^{-1}\] is clearly analytic and is a solution of the differential equation  $1+\beta z {q_{\beta}'(z)}/{q_{\beta}^2(z)} =\sqrt{1+z}$. Define the functions $\nu(w)=1$ and $\psi(w)=\beta/w^2 $. The function $Q:\overline{\mathbb{D}}\to\mathbb{C}$ defined by $Q(z)= z q_{\beta}'(z) \psi(q_{\beta}(z))=\beta z q_{\beta}'(z)/q_{\beta}^2(z)=\sqrt{1+z}-1$  is starlike in $\mathbb{D}$, $Q$ is starlike function. The function  $h(z):= \nu(q_{\beta}(z))+Q(z)=\nu(q_{\beta}(z))+Q(z)$ satisfies the inequality $\RE (zh'(z)/Q(z))>0$ for $ z \in\mathbb{D}$. Therefore, by using Lemma \ref{lemma1},  we see that the subordination \[1+ \beta  \frac{zp'(z)}{p^2(z)} \prec 1+ \beta \frac{z q_{\beta}'(z)}{q_{\beta}^2(z)}\] implies $p(z)\prec q_{\beta}(z)$. As the similar  lines of the proof of Theorem \ref{lemth1}, the proofs of  parts (a)-(d) are obtained.
\qedhere
\end{proof}

The subordination in part (d) of Theorem \ref{thd4} was earlier investigated in  \cite[Lemma 2.4, p.\ 6]{svr}  where non-sharp result was obtained.

 \end{document}